\documentclass{amsart}

\usepackage{amssymb, amsmath, amsthm}
\usepackage{mathrsfs}
\usepackage{mathabx}
\usepackage{verbatim}
\usepackage[shortlabels]{enumitem}
\usepackage[numbers,sort&compress]{natbib}
\usepackage[fleqn,tbtags]{mathtools}
\usepackage[colorlinks,linkcolor={blue},citecolor={blue},urlcolor={red},]{hyperref}

\let\le\leqslant

\let\ge\geqslant

\theoremstyle{plain}
\newtheorem{theorem}{Theorem}[section]
\theoremstyle{remark}
\newtheorem{remark}[theorem]{Remark}
\newtheorem{example}[theorem]{Example}
\theoremstyle{plain}
\newtheorem{corollary}[theorem]{Corollary}

\newtheorem{proposition}[theorem]{Proposition}
\newtheorem{definition}[theorem]{Definition}

\newtheorem{assumption}[theorem]{Assumption}

\numberwithin{equation}{section}

\newcommand{\E}{\mathbb E}

\newcommand{\R}{\mathbb R}
\newcommand{\scrH}{\mathscr H}

\newcommand{\scrL}{\mathscr L}
\newcommand{\Ran}{\operatorname{Ran}}
\newcommand{\Ker}{\operatorname{Ker}}
\newcommand{\Dom}{\operatorname{D}}

\newcommand{\dd}{\,{\rm d}}
\newcommand{\iprod}[2]{( #1|#2 )}
\newcommand{\lb}{\langle}
\newcommand{\rb}{\rangle}
\newcommand{\one}{{\bf 1}}
\newcommand{\Om}{\Omega}

\newcommand{\Tr}{{\rm Tr}}

\title[A Bismut--Elworthy--Li formula]{A Bismut--Elworthy--Li formula for linear stochastic evolution equations in Banach spaces}
\author{Jan van Neerven}
\date{\today}

\dedicatory{Dedicated to Bohdan Maslowski on the occasion of his 70th birthday}

\begin{document}

\begin{abstract}
We prove a Bismut--Elworthy--Li formula for linear stochastic evolution equations in Banach spaces. The admissible directions are those for which the deterministic orbit is square integrable with values in the Hilbertian range of the noise coefficient. Under this square function condition the derivative of the semigroup is represented by a stochastic integral against an explicit deterministic control. The result is a concrete control realisation of the Cameron--Martin formula for the fixed-time Gaussian transition measures.
\end{abstract}

\date{\today}



\maketitle

\section{Introduction}\label{sec:introduction}

Bismut--Elworthy--Li formulas express derivatives of Markov transition
semigroups without differentiating the test function. In their simplest form
they represent a derivative of \(P_tf\) as an expectation of \(f\) multiplied by
a stochastic integral. Such formulas go back to Bismut's integration by parts
method and to the probabilistic derivative formulas of Elworthy and Li; see
\cite{Bismut1984, ElworthyLi1994}. In infinite dimensions they play an
important role in the study of smoothing properties of transition semigroups,
in particular in connection with strong Feller estimates and uniqueness
questions for invariant measures; see, for instance,
\cite{DaPratoZabczyk1992, DaPratoZabczyk1996, PeszatZabczyk1995}.

In the Hilbert space theory of Ornstein--Uhlenbeck semigroups, such smoothing
questions are closely tied to Cameron--Martin spaces, invariant Gaussian
measures, and controllability properties. Strong Feller and gradient estimates
for Hilbert space Ornstein--Uhlenbeck semigroups have been studied extensively;
see, for instance,
\cite{ChojnowskaMichalikGoldys2002, GoldysVanNeerven2003, PeszatZabczyk1995}.
Related links between equivalence of Gaussian transition laws and
null-controllability appear in \cite{MaslowskiVanNeerven2012}.

The aim of the present paper is to give, in the setting of linear stochastic evolution equations with additive noise, a concrete square function condition which produces an explicit stochastic integral weight in the Bismut--Elworthy--Li formula.

Let \(S=(S(t))_{t\ge0}\) be a strongly continuous semigroup on a real Banach space \(E\) with generator \(A\), and consider the linear stochastic evolution equation

\begin{equation}\label{eq:ACP}
\left\{
\begin{aligned}
 \dd U_t & = AU_t\dd t+B\dd W_{H}(t),\quad t\ge 0,
 \\ U_0 & = x.
\end{aligned}
\right.
\end{equation}
Here, \(W_{H}\) is an \(H\)-cylindrical Brownian motion over a real separable Hilbert space \(H\), the operator \(B\in\scrL(H,E)\) is bounded and linear, and the initial condition \(x\) is an element of \(E\). Under the stochastic integrability assumption stated in Assumption \ref{ass:fixed-time-integrability}, the stochastic convolution
\[
    W_A(t)=\int_0^t S(t-s)B\,\dd W_{H}(s)
\]
is well defined as an \(E\)-valued centred Gaussian random variable. With this
notation, the mild solution of \eqref{eq:ACP} is given by
\[
    U_t^x=S(t)x+W_A(t), \quad t\ge 0.
\]
The process \(U^x\) is Markovian, and its transition semigroup on the space \(B_{\rm b}(E)\) of bounded Borel functions on \(E\), given by
\[
    P_tf(x)=\E f(U_t^x), \quad t\ge 0,
    \ x\in E,
\]
 is called the {\em Ornstein--Uhlenbeck semigroup} associated with \(A\) and \(B\).

The proof is guided by the following observation.  If
\(h\in L^2(0,t;H)\) is a deterministic function and
\[
    R_t h:=\int_0^t S(t-s)Bh(s)\dd s,
\]
then the condition
\[
    R_th=S(t)\xi
\]
means that the deterministic shift \(S(t)\xi\) belongs to the
Cameron--Martin space of the Gaussian random variable \(W_A(t)\).  For
\(g(y):=f(S(t)x+y)\) one has
\[
    P_tf(x+\varepsilon\xi)
    =
    \E g(W_A(t)+\varepsilon R_th).
\]
The Cameron--Martin differentiation formula therefore gives
\[
    D_\xi P_tf(x)
    =
    \E[f(U_t^x)W_t(h)]
    =
    \E\Bigl[
    f(U_t^x)
    \int_0^t\iprod{h(s)}{\dd W_H(s)}_H
    \Bigr],
\]
where \(W_t\) is the isonormal process over \(L^2(0,t;H)\) induced by
\(W_H\).  Thus one is led to choose \(h\) so that \(R_th=S(t)\xi\).
If, naively, \(B\)
had an inverse and if \(B^{-1}S(\cdot)\xi\) were square integrable, then for any
\(a\in W^{1,\infty}(0,t)\) with \(a(0+)=0\) and \(a(t-)=1\) the choice
\[
    h_\xi(s)=a'(s)B^{-1}S(s)\xi
\]
would give
\[
    R_th_\xi = \int_0^t a'(s)S(t-s)BB^{-1}S(s)\xi\dd s = S(t)\xi.
\]

Of course, in the Banach space setting \(B\) need be neither injective nor
surjective, and the expression \(B^{-1}S(s)\xi\) has no literal meaning in
general. The main point of the paper is to make this computation
precise by replacing the inverse of \(B\) with the inverse of \(B\) on its
Hilbertian range, and by identifying those directions for which
the resulting deterministic control is square integrable. This is done in
Section \ref{sec:noise-range-square-function}, where the space \(E_B\), the
operator \(B^\dagger\), and the space of admissible directions \(\mathscr D_t\) are introduced.
The control identity resulting from this construction is
\[
    \int_0^t S(t-s)Bh_\xi(s)\dd s=S(t)\xi, \quad \xi\in\mathscr D_t.
\]

The appearance of a deterministic control problem is not accidental. In the
linear Gaussian setting, smoothing properties of transition semigroups are
closely related to controllability properties; see,
for instance, \cite{MaslowskiVanNeerven2012}. The control interpretation of
Bismut-type formulae is also central in the recent work of Goldys and Peszat
\cite{GoldysPeszat2024}, where null-controllability is used to obtain
probabilistic formulae for derivatives of transition semigroups of generalised
Ornstein--Uhlenbeck processes. Here we identify directly the directions for
which the formal expression \(B^{-1}S(\cdot)\xi\) can be made rigorous through
the Hilbertian range of \(B\).

The main result, proved in Section \ref{sec:linear-bel-formula}, is the
following. If \(f:E\to\R\) is bounded and Borel measurable, \(x\in E\), and
\(\xi\in\mathscr D_t\), then \(P_tf\) is differentiable at \(x\) in the
direction \(\xi\), and
\begin{equation}\label{eq:linear-bel-intro}
    D_\xi P_tf(x) = \E\Bigl[f(U_t^x) \int_0^t \iprod{a'(s)B^\dagger S(s)\xi}{\dd W_{H}(s)}_{H} \Bigr].
\end{equation}
For the choice \(a(s)=s/t\), this gives the pointwise gradient bound
\[
    |D_\xi P_tf(x)| \le \frac{\sqrt{2/\pi}}{t}\,\|f\|_\infty
    \Bigl( \int_0^t\|S(s)\xi\|_{E_B}^2\dd s \Bigr)^{1/2}.
\]

The relation with the classical reproducing kernel Hilbert space approach is as
follows. If \(R_t:L^2(0,t;H)\to E\) is the deterministic convolution operator introduced above, the Cameron--Martin space \(H_t\) of the law of \(W_A(t)\) is the range of \(R_t\), equipped with its quotient Hilbert norm. The condition
\(S(t)\xi\in H_t\) is the classical Gaussian condition behind differentiability
in the direction \(\xi\). The square function condition \(\xi\in\mathscr D_t\) gives a concrete sufficient condition for this inclusion, and at the same time provides the explicit stochastic integral appearing in \eqref{eq:linear-bel-intro}. This connection is recalled in Section \ref{sec:relation-rkhs-criterion}; see also \cite{DaPratoZabczyk1992, GoldysVanNeerven2003, vanNeerven1998}.

In Section \ref{sec:examples-first-consequences} we work out two examples.
In the first example, \(S(t)=e^{-tA}\) is the \(C_0\)-semigroup generated by \(-A\), where \(A\) is a non-negative self-adjoint operator acting in a Hilbert space \(E\). Taking the noise operator to be \(B=(I+A)^{-\alpha}\), the range condition becomes
\[
    \int_0^t \|(I+A)^\alpha S(s)\xi\|^2\dd s<\infty.
\]
This holds for all \(\xi\in E\) when \(0\le\alpha\le1/2\), and on a natural
fractional domain when \(\alpha>1/2\). The Bismut--Elworthy--Li formula, with \(a(s) = s/t\), leads to the pointwise gradient bound
\begin{align*} \ \hskip.9cm
    |D_\xi P_tf(x)|
    & \le \frac{\sqrt{2/\pi}}{t}\,\|f\|_\infty
    \Bigl(\int_0^t \|(I+A)^\alpha e^{-sA}\xi\|^2\dd s \Bigr)^{1/2}.
\intertext{Since \(\mathscr D_t =E\) for \(0\le\alpha\le1/2\), in this parameter range the semigroup \(P_t\) maps bounded Borel functions to Lipschitz continuous functions on \(E\) for every \(t>0\).
\newline\indent
The second example is the Dirichlet heat semigroup on \(L^p(0,1)\),
\(1\le p\le2\), driven by \(L^2(0,1)\)-cylindrical noise through the natural
inclusion \(L^2(0,1)\hookrightarrow L^p(0,1)\). In this case the Hilbertian
noise range is \(L^2(0,1)\), and standard heat kernel estimates imply that
\(S_p(\cdot)\xi\in L^2(0,t;L^2(0,1))\) for every \(\xi\in L^p(0,1)\). Hence
\(\mathscr D_t=L^p(0,1)\), and the Bismut--Elworthy--Li formula gives the pointwise bound
     }
    |D_\xi P_tf(x)|
    & \le C_p\,t^{-1/4-1/(2p)} \|f\|_\infty\|\xi\|_{L^p(0,1)}.
\end{align*}
Thus the heat example gives full Lipschitz regularisation of \(P_t\) on
\(L^p(0,1)\) for every \(t>0\) and every \(1\le p\le2\).

\medskip
The paper is organised as follows. Section \ref{sec:preliminaries} recalls the
Gaussian preliminaries and the stochastic integration notation. Section \ref{sec:linear-ou-equation}
sets up the linear Ornstein--Uhlenbeck equation. Section \ref{sec:noise-range-square-function} introduces \(E_B\) and \(\mathscr D_t\). Section \ref{sec:linear-bel-formula} proves the exact-control identity and the Bismut--Elworthy--Li formula. Section \ref{sec:examples-first-consequences} contains examples and first consequences, and Section \ref{sec:relation-rkhs-criterion} relates the result to the RKHS approach.

\section{Preliminaries}\label{sec:preliminaries}

In this section we fix terminology and collect some standard results needed later on in the paper. We will always assume that \(E\) is a real Banach space and \(H\) a real separable Hilbert space.

\subsection{\texorpdfstring{\(\gamma\)-Radonifying operators}{gamma-Radonifying operators}}

In this section we briefly recall the definition and relevant properties of \(\gamma\)-radonifying operators. Full accounts include \cite{HNVW-volume2, vanNeerven2010}.

Let \((\gamma_n)_{n\ge1}\) be a sequence of independent standard real Gaussian random variables. A bounded operator \(R:H\to E\) is called \emph{\(\gamma\)-radonifying} if, for one, equivalently for every, orthonormal basis \((h_n)_{n\ge1}\) of \(H\), the Gaussian series
\[
    \sum_{n\ge1}\gamma_nRh_n
\]
converges in \(L^2(\Omega;E)\). With respect to the norm
\[
    \|R\|_{\gamma(H,E)} :=
    \Bigl(\E\Bigl\|\sum_{n\ge1}\gamma_nRh_n\Bigr\|^2 \Bigr)^{1/2},
\]
the space \(\gamma(H,E)\) is a Banach space.

We have \(\gamma(H,\R) = H\) canonically. If \(E\) is a Hilbert space, then \(\gamma(H,E)\) coincides isometrically with the space of Hilbert--Schmidt operators from \(H\) into \(E\).
If \((S,\mu)\) is a measure space, \(1\le p<\infty\), and \(F\) is a Banach space, then the Kahane--Khintchine inequalities imply the isomorphism \[ \gamma(H,L^p(S,\mu;F)) \simeq L^p(S,\mu;\gamma(H,F)). \]

The class of \(\gamma\)-radonifying operators is stable under left and right multiplication with bounded operators:

\begin{proposition}[Ideal property]\label{prop:gamma-ideal-property}
Let \(H_1,H_2\) be separable Hilbert spaces and let \(E,F\) be Banach spaces. If
\(R\in\gamma(H_1,E)\), \(T\in\scrL(H_2,H_1)\), and \(L\in\scrL(E,F)\), then
\(LRT\in\gamma(H_2,F)\) and
\[
    \|LRT\|_{\gamma(H_2,F)} \le \|L\|_{\scrL(E,F)} \|R\|_{\gamma(H_1,E)} \|T\|_{\scrL(H_2,H_1)}.
\]
\end{proposition}

This standard ideal property of \(\gamma\)-radonifying operators will be used
without further comment.

\subsection{Isonormal Gaussian processes and cylindrical Brownian motion} \label{subsec:isonormal-processes}

We shall use two closely related ways of representing Gaussian noise.

Let \(\scrH\) be a real separable Hilbert space. An \emph{\(\scrH\)-isonormal Gaussian process} is a bounded linear operator
\[
    W:\scrH\to L^2(\Omega)
\]
such that, for all \(h_1,\ldots,h_N\in \scrH\), the random vector
\((W h_1,\ldots,W h_N)\) is centred Gaussian and
\[
    \E(W h_1\,W h_2) = \iprod{h_1}{h_2}_{\scrH}, \quad h_1,h_2\in \scrH.
\]
For more details, and the connection to Malliavin Calculus, we refer to \cite{Nualart2006}.

Let \(H\) be a real separable Hilbert space. An \emph{\(H\)-cylindrical Brownian
motion} is a family
\[
    W_H(t):H\to L^2(\Omega), \quad t\ge0,
\]
of bounded linear operators such that, for all \(N\ge1\), all \(h_1,\ldots,h_N\in H\), and all \(t_1,\ldots,t_N\ge0\), the random vector \((W_H(t_1)h_1,\ldots,W_H(t_N)h_N)\) is centred Gaussian and
\[
    \E\bigl(W_H(t_i)h_i\,W_H(t_j)h_j\bigr)
    = (t_i\wedge t_j)\iprod{h_i}{h_j}_H, \quad 1\le i,j\le N.
\]
For each \(h\in H\), the process \(t\mapsto W_H(t)h\) is a real-valued
Brownian motion with variance \(\|h\|_H^2\).

The two notions are related in the following way. If \(W_H\) is an \(H\)-cylindrical Brownian motion, then the formula
\[
    W\bigl(\one_{(a,b]}\otimes h\bigr)
    := W_H(b)h-W_H(a)h, \quad 0\le a<b<\infty,\ h\in H,
\]
extends uniquely to an \(L^2(\R_+;H)\)-isonormal Gaussian process. For this extension we use the notation
\begin{align}\label{eq:Wh}
    W(f) = \int_0^\infty \iprod{f(s)}{\dd W_H(s)}_H, \quad f\in L^2(\R_+;H).
\end{align}
When \(f\in L^2(0,t;H)\), we regard
\(f\) as an element of \(L^2(\R_+;H)\) by extending it by zero outside
\((0,t)\), and write
\[
    W(f) = \int_0^t \iprod{f(s)}{\dd W_H(s)}_H.
\]
Conversely, if \(W\) is an \(L^2(\R_+;H)\)-isonormal Gaussian process, then
\[
    W_H(t)h := W(\one_{(0,t]}\otimes h), \quad t\ge 0,\ h\in H,
\]
defines an \(H\)-cylindrical Brownian motion.

\subsection{The Cameron--Martin formula}\label{subsec:cameron-martin-gaussian}

If \(R\in\gamma(\scrH,E)\) and \(W\) is an
isonormal Gaussian process over \(\scrH\), we denote by \(X_R\) the centred
\(E\)-valued Gaussian random variable canonically associated with \(R\),
characterised by
\[
    \lb X_R, x^*\rb = W(R^*x^*), \quad x^*\in E^*.
\]
Such a random variable exists and is unique up to a set of measure zero. In fact, for any orthonormal basis \((e_n)_{n\ge1}\)
of \(\scrH\), the series
\begin{align}\label{eq:XR}
   X_R = \sum_{n\ge1} W(e_n)Re_n
\end{align}
converges in \(L^2(\Omega;E)\) since \(R\) is \(\gamma\)-radonifying, is independent of the chosen basis, and has the desired properties.

We need the Cameron--Martin formula in the following standard form; see \cite[Corollary~2.4.3]{Bogachev1998}.

\begin{theorem}[Cameron--Martin formula] \label{thm:cameron-martin-formula}
Let \(R\in\gamma(\scrH,E)\), let \(W\) be an \(\scrH\)-isonormal Gaussian
process, and let \(X_R\) be the centred \(E\)-valued Gaussian random variable
canonically associated with \(R\). Then, for every bounded Borel function
\(g:E\to\R\), every \(h\in\scrH\), and every \(\varepsilon\in\R\),
\begin{equation}\label{eq:cameron-martin-formula-epsilon}
    \E g(X_R+\varepsilon Rh)
    = \E\Bigl[ g(X_R)\exp\Bigl(\varepsilon W(h) -\frac{\varepsilon^2}{2}\|h\|_\scrH^2\Bigr)\Bigr].
\end{equation}
\end{theorem}

\begin{remark}\label{rem:cameron-martin-kernel}
If \(Rh=0\), then \[\E g(X_R+\varepsilon Rh)=\E g(X_R)\] and therefore the right-hand side of
\eqref{eq:cameron-martin-formula-epsilon} is independent of \(\varepsilon\).
This can also be seen directly as follows.
Let \(x^*\in E^*\). By the defining property of \(X_R\) we have \(\lb X_R,x^*\rb = W(R^*x^*)\), and therefore
\[
    \E\bigl[\lb X_R,x^*\rb W(h)\bigr]
    = \E\bigl[W(R^*x^*)W(h)\bigr]
    = \iprod{R^*x^*}{h}_\scrH
    = \lb Rh,x^*\rb
    = 0.
\]
Thus \(W(h)\) is orthogonal in \(L^2(\Om)\) to every scalar random variable \(\lb X_R,x^*\rb\). More generally, if \(x_1^*,\ldots,x_n^*\in E^*\), then
\[
    \bigl(W(h),\lb X_R,x_1^*\rb,\ldots,\lb X_R,x_n^*\rb\bigr)
\]
is a centred Gaussian vector. The preceding computation shows that the first coordinate is
uncorrelated with each of the remaining coordinates. Hence, by the elementary independence criterion for Gaussian vectors, \(W(h)\) is independent of
\[
    \bigl(\lb X_R,x_1^*\rb,\ldots,\lb X_R,x_n^*\rb\bigr).
\]
A monotone class argument now gives independence of \(W(h)\) and the \(\sigma\)-algebra generated by the scalar random variables \(\lb X_R,x^*\rb\), \(x^*\in E^*\). Since \(X_R\) is strongly measurable, this is the \(\sigma\)-algebra generated by \(X_R\), up to completion. Hence \(W(h)\) is independent of \(X_R\).

Consequently, for every bounded Borel function \(g:E\to\R\),
\begin{align*}
    \E\Bigl[ g(X_R)\exp\Bigl(\varepsilon W(h) -\frac{\varepsilon^2}{2}\|h\|_\scrH^2\Bigr) \Bigr]
    & = \E g(X_R)\, \E\exp\Bigl(\varepsilon W(h) -\frac{\varepsilon^2}{2}\|h\|_\scrH^2\Bigr)
    \\ & = \E g(X_R),
\end{align*}
independently of \(\varepsilon\).
\end{remark}

As an application of Theorem \ref{thm:cameron-martin-formula} we have the following result.

\begin{corollary}[Differentiation in the Cameron--Martin space directions]
\label{cor:cameron-martin-differentiation}
Let \(R\in \gamma(\scrH,E)\), and let \(g:E\to\R\) be bounded and Borel measurable. Then for every \(h\in \scrH\) the map
\[
    \varepsilon\mapsto \E g(X_R+\varepsilon Rh)
\]
is differentiable at \(0\) and
\begin{equation}\label{eq:cameron-martin-differentiation}
    \Bigl.\frac{\rm d}{{\rm d}\varepsilon}\Bigr|_{\varepsilon=0}
    \E g(X_R+\varepsilon Rh) = \E[g(X_R)W(h)].
\end{equation}
\end{corollary}

\begin{proof}
By Theorem \ref{thm:cameron-martin-formula},
\[
    \E g(X_R+\varepsilon Rh) = \E[g(X_R)M_\varepsilon],
    \qquad M_\varepsilon := \exp\Bigl(\varepsilon W(h)
    -\frac{\varepsilon^2}{2}\|h\|_\scrH^2\Bigr).
\]
We claim that for \(0<|\varepsilon|\le1\) the difference quotients
\((M_\varepsilon-1)/\varepsilon\) are dominated in \(L^1(\Omega)\) by an integrable random variable depending only on \(W(h)\) and \(\|h\|_\scrH\). To prove this, put \(Z:=W(h)\) and \(\sigma:=\|h\|_\scrH\), so that \(M_\varepsilon=\exp(\varepsilon Z-\frac{\varepsilon^2}{2}\sigma^2)\). For \(0<|\varepsilon|\le1\), the mean value theorem gives
\begin{align*}
    \Bigl|\frac{M_\varepsilon-1}{\varepsilon}\Bigr|
    & \le \sup_{|r|\le1} |Z-r\sigma^2| \exp\Bigl(rZ-\frac{r^2}{2}\sigma^2\Bigr)
    \le (|Z|+\sigma^2) \exp\Bigl(|Z|+\frac12\sigma^2\Bigr).
\end{align*}
The right-hand side is integrable by Fernique's theorem. This proves the claim.

Since
\[
    \Bigl.\frac{\rm d}{{\rm d}\varepsilon}\Bigr|_{\varepsilon=0}M_\varepsilon
    = \Bigl((Z-\varepsilon\sigma^2)M_\varepsilon\Bigr)\Bigr|_{\varepsilon=0} = Z= W(h),
\]
dominated convergence now gives \eqref{eq:cameron-martin-differentiation}.
\end{proof}

\section{The linear Ornstein--Uhlenbeck equation}
\label{sec:linear-ou-equation}

A function \(\Phi:\R_+\to\scrL(H,E)\)
is said to be {\em stochastically integrable} with respect to a \(H\)-cylindrical Brownian motion \(W_{H}\) if
\(t\mapsto \lb \Phi(t)h,x^*\rb\) is square integrable for all \(h\in H\) and \(x^*\in E^*\) and \(\Phi\) is represented by an operator
\(I_\Phi\in\gamma(L^2(\R_+;H),E)\)
in the sense that, for all \(0\le a<b<\infty\), \(h\in H\), and
\(x^*\in E^*\),
\[
    \bigl\lb I_\Phi(\one_{[a,b)}\otimes h),x^*\bigr\rb
    =
    \int_a^b
    \bigl\lb \Phi(s)h,x^*\bigr\rb\dd s.
\]
In that case
\[
    \int_0^\infty \Phi(s)\dd W_{H}(s)
\]
denotes the centred \(E\)-valued Gaussian random variable canonically
associated with the operator \(I_\Phi\).
Stochastic integrability on finite intervals is defined by extending the
integrand by zero outside the interval.

Let \(S=(S(t))_{t\ge0}\) be a \(C_0\)-semigroup on \(E\) with generator
\(A\), and let \(B\in\scrL(H,E)\).
The following assumption will be in place in the remainder of the paper.

\begin{assumption}[Stochastic integrability]
\label{ass:fixed-time-integrability}
The function
\[
    s\mapsto S(s)B
\]
is stochastically integrable with respect to \(W_{H}\) on every interval \([0,t]\).
\end{assumption}

This assumption easily implies that for all \(t>0\) the function
\(s\mapsto S(t-s)B\)
is stochastically integrable with respect to \(W_{H}\) on \([0,t]\), and consequently the \(E\)-valued Gaussian
random variable
\[
    W_A(t):=
    \int_0^t S(t-s)B\,\dd W_H(s)
\]
is well defined.
For \(t>0\) we put
\[
    \scrH_t:=L^2(0,t;H).
\]
By the definition of stochastic integrability, the \(\gamma\)-radonifying operator from \(\scrH_t\) to \(E\) associated with this stochastic integral is precisely the convolution operator
\[R_t v := \int_0^t S(t-s)Bv(s)\dd s.\]
Thus,
\begin{align}\label{eq:WARt}
W_A(t)=X_{R_t}.
\end{align}

As in the Introduction we now define,
for \(t\ge 0\) and \(x\in E\),
\begin{equation}\label{eq:linear-ou-solution}
    U_t^x
    :=
    S(t)x+W_A(t)
    =
    S(t)x+\int_0^t S(t-s)B\dd W_{H}(s).
\end{equation}

We write \(C_{\rm b}^1(E)\) for the space of bounded Fr\'echet differentiable
functions \(f:E\to\R\) whose derivative \(Df:E\to E^*\) is bounded and
continuous.
We write \(C_{\rm b}^{1,u}(E)\) for the subspace of \(C_{\rm b}^1(E)\)
consisting of those functions whose derivative \(Df:E\to E^*\) is uniformly
continuous.

As in the introduction we define \(P\) as the transition semigroup of the Markovian process \(U\), that is, for \(t\ge 0\) and \(x\in E\) we define
\[ P_t f(x) := \E f(U_t^x).\]
The next proposition describes the directional derivatives of \(P\).

\begin{proposition}[Directional derivatives]
\label{prop:linear-directional-differentiability-smooth}
If \(f\in C_{\rm b}^1(E)\),
then, for all \(x,\xi\in E\), the directional derivative
\[
    D_\xi P_tf(x):=
    \lim_{\varepsilon\to0}
    \frac1{\varepsilon}(P_tf(x+\varepsilon\xi)-P_tf(x))
\]
exists and is given by
\begin{equation}\label{eq:linear-directional-derivative-formula}
    D_\xi P_tf(x)
    =
    \E\,\lb S(t)\xi, Df(U_t^x)\rb.
\end{equation}
If \(f\in C_{\rm b}^{1,u}(E)\), then \(P_tf\in C_{\rm b}^1(E)\) and
\begin{equation}\label{eq:linear-frechet-derivative-formula}
    \lb \xi, DP_tf(x)\rb
    =
    \E\,\lb S(t)\xi, Df(U_t^x)\rb.
\end{equation}
\end{proposition}

\begin{proof}
For \(\varepsilon\ne0\), by the fundamental theorem of calculus we have
\[
\begin{aligned}
    \frac1{\varepsilon}(P_tf(x+\varepsilon\xi)-P_tf(x))
    & =
    \E\Bigl[
    \frac1{\varepsilon}(f(U_t^x+\varepsilon S(t)\xi)-f(U_t^x))\Bigr]
    \\ &=
    \E\int_0^1
    \lb S(t)\xi, Df(U_t^x+r\varepsilon S(t)\xi)\rb \dd r,
\end{aligned}
\]
and by continuity of \(Df\), dominated convergence gives
\eqref{eq:linear-directional-derivative-formula}.

Assume now that \(f\in C_{\rm b}^{1,u}(E)\). For \(x\in E\) define
\(L_x\in E^*\) by
\[
    L_x\xi
    :=
    \E\,\lb S(t)\xi, Df(U_t^x)\rb ,
    \quad \xi\in E.
\]
Then \(L_x\) is a bounded linear functional on \(E\) with norm
\(
\|L_x\| \le \|Df\|_\infty\|S(t)\| .
\)
We claim that \(L_x\) is the Fr\'echet derivative of \(P_tf\) at \(x\).

Indeed, for \(\xi\in E\), using the identity \(U_t^{x+\xi}=U_t^x+S(t)\xi\) and reasoning as before,
\begin{align*}
|P_tf(x+\xi)-P_tf(x)-L_x\xi|
& = \Bigl|\E\int_0^1\lb S(t)\xi, Df(U_t^x+rS(t)\xi)-Df(U_t^x) \rb \dd r \Bigr|
\\ & \quad \le
\|S(t)\|\|\xi\|
\sup_{\|y-z\|\le \|S(t)\|\|\xi\|}
\|Df(y)-Df(z)\|.
\end{align*}
Since \(Df\) is uniformly continuous, the right-hand side supremum tends to \(0\) as
\(\|\xi\|\to0\). Therefore
\[
  \lim_{\xi\to 0} \, \frac1{\|\xi\|}
    |P_tf(x+\xi)-P_tf(x)-L_x\xi| = 0 .
\]
This proves the claim.

It remains only to check that \(x\mapsto DP_tf(x)\) is continuous as a
map from \(E\) into \(E^*\). Reasoning as above, if \(x,x'\in E\), then, for \(\|\xi\|\le1\) we have
\begin{align*}
\bigl|\lb \xi,DP_tf(x)-DP_tf(x')\rb\bigr|
& =
\bigl|\E\bigl[\lb S(t)\xi, Df(U_t^x)-Df(U_t^{x'})\rb\bigr]\bigr|
\\ &\le
\|S(t)\|\sup_{\|y-z\|\le\|S(t)\|\|x-x'\|} \|Df(y)-Df(z)\|.
\end{align*}
Taking the supremum over \(\|\xi\|\le1\) and using the uniform continuity of
\(Df\), we obtain
\[
    \|DP_tf(x)-DP_tf(x')\|\to 0
    \ \ \hbox{as } x'\to x .
\]
Thus \(DP_tf\) is continuous. The boundedness of \(DP_tf\) follows from
\(\|DP_tf(x)\|\le \|Df\|_\infty\|S(t)\|\).
\end{proof}

\section{The square function condition}
\label{sec:noise-range-square-function}

Associated with the noise operator \(B\) we now define the space
\[
    E_B:=\Ran (B).
\]
We equip \(E_B\) with the Hilbert norm transported from
\((\Ker (B))^\perp\subseteq H\).
With this norm \(E_B\) is a Hilbert space, and the inclusion \(E_B\hookrightarrow E\)
is continuous, since
\[
    \|y\|\le \|B\|_{\scrL(H,E)}\|y\|_{E_B},
    \quad y\in E_B.
\]
The operator \(B\) induces an isometry from
\((\Ker(B))^\perp\) onto \(E_B\). We denote its inverse by
\[
    B^\dagger:E_B\to(\Ker(B))^\perp\subseteq H.
\]

For a direction \(\xi\in E\), we shall
need the orbit \(s\mapsto S(s)\xi\)
to take its values in \(E_B\) for almost all \(s\in(0,t)\) in a square integrable way with respect to the Hilbert norm of \(E_B\).

\begin{definition}[Admissible directions]
\label{def:linear-admissible-directions}
For \(t>0\), define
\[
    \mathscr D_t
    :=
    \Bigl\{
    \xi\in E:\ S(\cdot)\xi\in L^2(0,t;E_B)
    \Bigr\}.
\]
For \(\xi\in\mathscr D_t\) we write
\[
    \|\xi\|_{\mathscr D_t}
    :=
    \Bigl(\int_0^t\|S(s)\xi\|_{E_B}^2\dd s\Bigr)^{1/2}.
\]
\end{definition}

By definition, we have \(\xi\in\mathscr D_t\) if and only if \(B^\dagger S(\cdot)\xi\in L^2(0,t;H)\),
and in that case
\[
    \|\xi\|_{\mathscr D_t} = \|B^\dagger S(\cdot)\xi\|_{L^2(0,t;H)}.
\]
If \(\|\xi\|_{\mathscr D_t}=0\), then \(S(s)\xi=0\) for almost all
\(s\in(0,t)\). By strong continuity this holds for all \(s\in(0,t)\), and
letting \(s\downarrow0\) gives \(\xi=0\). Thus
\(\|\cdot\|_{\mathscr D_t}\) is a norm on \(\mathscr D_t\).

\section{The linear Bismut--Elworthy--Li formula}
\label{sec:linear-bel-formula}

Let \(t>0\). Let \(a\in W^{1,\infty}(0,t)\) satisfy
\begin{equation}\label{eq:a-boundary-conditions}
    a(0+)=0,
    \qquad
    a(t-)=1.
\end{equation}
For \(\xi\in\mathscr D_t\) define
\begin{equation}\label{eq:linear-exact-control}
    v_a^\xi(s):=a'(s)B^\dagger S(s)\xi,
    \quad 0<s<t.
\end{equation}
Then \(v_a^\xi\in \scrH_t=L^2(0,t;H)\), and
\begin{equation}\label{eq:linear-control-norm}
    \|v_a^\xi\|_{\scrH_t}
    \le
    \|a'\|_{L^\infty(0,t)}
    \Bigl(
    \int_0^t\|S(s)\xi\|_{E_B}^2\dd s
    \Bigr)^{1/2}.
\end{equation}

The following identity is the deterministic core of the argument.

\begin{proposition}[Exact control identity]
\label{prop:linear-exact-control}
Let \(\xi\in\mathscr D_t\). Then
\begin{equation}\label{eq:exact-control-identity}
    R_tv_a^\xi=S(t)\xi.
\end{equation}
\end{proposition}

\begin{proof}
The map \(s\mapsto Bv_a^\xi(s)\) belongs to \(L^2(0,t;E)\). Since \(S\) is
bounded on \([0,t]\), the function
\[
    s\mapsto S(t-s)Bv_a^\xi(s), \quad 0<s<t,
\]
belongs to \(L^1(0,t;E)\). Therefore the Bochner integral defining
\(R_tv_a^\xi\) is well defined.

For almost all \(s\in(0,t)\), \(S(s)\xi\in E_B\). Hence
\(BB^\dagger S(s)\xi=S(s)\xi\), and
\[
    S(t-s)Bv_a^\xi(s)
    =
    a'(s)S(t-s)S(s)\xi
    =
    a'(s)S(t)\xi.
\]
It follows that
\[
\begin{aligned}
    R_tv_a^\xi
    =
    \int_0^t S(t-s)Bv_a^\xi(s)\dd s
    =
    \int_0^t a'(s)S(t)\xi\dd s
    =
    \Bigl(\int_0^t a'(s)\dd s\Bigr)S(t)\xi
    =
    S(t)\xi.
\end{aligned}
\]
The last equality follows from the boundary conditions on \(a\).
\end{proof}

We now turn to the main formula, which comes in two forms. Since the formula applies to bounded Borel
functions, we first fix the meaning of the derivative. For bounded Borel
\(f:E\to\R\), \(x\in E\), and \(\xi\in E\), we write
\begin{equation}\label{eq:directional-derivative-convention}
    D_\xi P_tf(x)
    :=
    \Bigl.\frac{\rm d}{{\rm d}\varepsilon}\Bigr|_{\varepsilon=0}
    P_tf(x+\varepsilon\xi).
\end{equation}

\begin{theorem}[Abstract Bismut--Elworthy--Li formula]\label{thm:RKHS}
Let Assumption \ref{ass:fixed-time-integrability} hold, and fix \(t>0\).
Suppose that \(\xi\in E\) and \(h\in\scrH_t\) satisfy
\[
    S(t)\xi=R_th.
\]
Then, for every bounded Borel function \(f:E\to\R\), the derivative
\(D_\xi P_tf(x)\) exists and
\[
    D_\xi P_tf(x)
    =
    \E[f(U_t^x)W_t(h)]
    =
    \E\Bigl[
    f(U_t^x)
    \int_0^t \iprod{\Pi_t h(s)}{\dd W_{H}(s)}_{H}
    \Bigr],
\]
where \(W_t\) is the \(\scrH_t\)-isonormal process associated with \(W_H\) and
\(\Pi_t\) denotes the orthogonal projection of \(\scrH_t=L^2(0,t;H)\)
onto \((\Ker (R_t))^\perp\).
\end{theorem}

The first identity holds for any representing control \(h\), while the second
uses its minimal-norm representative \(\Pi_t h\) modulo \(\Ker(R_t)\).

\begin{proof}
Let \(W_t\) be the \(\scrH_t\)-isonormal process associated with the cylindrical
Brownian motion \(W_H\), and put \(g(y):=f(S(t)x+y)\) for \(y\in E\). Since
\(X_{R_t}=W_A(t)\), we have
\[
    P_tf(x+\varepsilon\xi)
    =
    \E f(S(t)x+W_A(t)+\varepsilon S(t)\xi)
    =
    \E g(X_{R_t}+\varepsilon R_th).
\]
The Cameron--Martin differentiation formula applied to \(R_t\) gives
\[
    D_\xi P_tf(x)
    =
    \E[g(X_{R_t})W_t(h)]
    =
    \E[f(U_t^x)W_t(h)].
\]
This proves the first identity. For the second identity, note that \(h-\Pi_t h\in\Ker(R_t)\). Hence
\[
    R_t\Pi_t h=R_th=S(t)\xi.
\]
Applying the first formula with \(\Pi_t h\) in place of \(h\), and using the
definition of \(W_t\), gives
\[
    D_\xi P_tf(x)
    =
    \E[f(U_t^x)W_t(\Pi_t h)]
    =
    \E\Bigl[
    f(U_t^x)
    \int_0^t \iprod{\Pi_t h(s)}{\dd W_H(s)}_H
    \Bigr].
\]
\end{proof}

\begin{theorem}[Explicit Bismut--Elworthy--Li formula]
\label{thm:linear-bel-formula}
Let Assumption \ref{ass:fixed-time-integrability} hold. Let \(x\in E\),
\(\xi\in\mathscr D_t\), and let \(a\in W^{1,\infty}(0,t)\) satisfy
\eqref{eq:a-boundary-conditions}. Then, for every bounded Borel function
\(f:E\to\R\), the derivative \(D_\xi P_tf(x)\) exists and
\begin{equation}\label{eq:linear-bel-formula}
    D_\xi P_tf(x)
    =
    \E\Bigl[
    f(U_t^x)
    \int_0^t
    \iprod{a'(s)B^\dagger S(s)\xi}{\dd W_{H}(s)}_{H}
    \Bigr].
\end{equation}
The stochastic integral is the real Gaussian random variable \(W_t(v_a^\xi)\).
If \(f\in C_{\rm b}^1(E)\), then the same derivative is also given by
\begin{equation}\label{eq:linear-bel-consistency-smooth}
    D_\xi P_tf(x)
    =
    \E\,\lb S(t)\xi,Df(U_t^x)\rb .
\end{equation}
\end{theorem}

\begin{proof}
By Proposition \ref{prop:linear-exact-control},
\[
    R_tv_a^\xi=S(t)\xi.
\]
Applying Theorem \ref{thm:RKHS} with \(h=v_a^\xi\) gives
\[
    D_\xi P_tf(x)
    =
    \E[f(U_t^x)W_t(v_a^\xi)].
\]
By the definition of the \(\scrH_t\)-isonormal process associated with \(W_H\),
\[
    W_t(v_a^\xi)
    =
    \int_0^t
    \iprod{v_a^\xi(s)}{\dd W_H(s)}_H .
\]
Since
\[
    v_a^\xi(s)=a'(s)B^\dagger S(s)\xi,
\]
this gives \eqref{eq:linear-bel-formula}.

If \(f\in C_{\rm b}^1(E)\), the identity
\[
    D_\xi P_tf(x)
    =
    \E\,\lb S(t)\xi,Df(U_t^x)\rb
\]
is precisely Proposition \ref{prop:linear-directional-differentiability-smooth}.
This proves \eqref{eq:linear-bel-consistency-smooth}.
\end{proof}

The formula gives the following gradient estimate:

\begin{corollary}[Directional gradient estimate]
\label{cor:linear-directional-gradient-estimate}
Under the assumptions of Theorem \ref{thm:linear-bel-formula},
\[
    |D_\xi P_tf(x)|
    \le
    \sqrt{2/\pi}\,\|f\|_\infty
    \|a'\|_{L^\infty(0,t)}
    \Bigl(
    \int_0^t\|S(s)\xi\|_{E_B}^2\dd s
    \Bigr)^{1/2}.
\]
In particular, for \(a(s)=s/t\),
\begin{equation}\label{eq:linear-gradient-estimate-linear-a}
    |D_\xi P_tf(x)|
    \le
    \frac{\sqrt{2/\pi}}{t}\,
    \|f\|_\infty
    \Bigl(
    \int_0^t\|S(s)\xi\|_{E_B}^2\dd s
    \Bigr)^{1/2}.
\end{equation}
\end{corollary}
\begin{proof}
By Theorem \ref{thm:linear-bel-formula},
\[
    |D_\xi P_tf(x)|
    \le
    \|f\|_\infty\,\E|W_t(v_a^\xi)|.
\]
The random variable \(W_t(v_a^\xi)\) is centred Gaussian with variance
\(\|v_a^\xi\|_{\scrH_t}^2\). Therefore
\[
    \E|W_t(v_a^\xi)|
    =
    \sqrt{2/\pi}\,\|v_a^\xi\|_{\scrH_t}.
\]
By \eqref{eq:linear-control-norm},
\[
    \|v_a^\xi\|_{\scrH_t}
    \le
    \|a'\|_{L^\infty(0,t)}
    \Bigl(
    \int_0^t\|S(s)\xi\|_{E_B}^2\dd s
    \Bigr)^{1/2}.
\]
This proves the first estimate. The second follows by taking \(a(s)=s/t\).
\end{proof}

\section{Examples}
\label{sec:examples-first-consequences}

We illustrate the square function condition by working out two examples:
a basic Hilbert space example and the Dirichlet heat equation in \(L^p(0,1)\).

\subsection{A self-adjoint Hilbert space example}
\label{subsec:self-adjoint-hilbert-example}

Let \(E\) be a real separable Hilbert space, and let \(A\) be a
non-negative self-adjoint operator on \(E\). We put
\[
    S(t):=e^{-tA},
    \quad t\ge0.
\]
By spectral theory, \(S\) is a \(C_0\)-semigroup of self-adjoint contractions
and its generator is \(-A\). Let \(\alpha\ge0\),
set \(H=E\), and take
\[
    B:=(I+A)^{-\alpha}.
\]
Then \(B\) is bounded, one-to-one, self-adjoint, and has dense range. The
Hilbertian noise range is
\[
    E_B
    =
    \Ran ((I+A)^{-\alpha})
    =
    \Dom((I+A)^\alpha),
\]
with norm
\[
    \|x\|_{E_B}
    =
    \|(I+A)^\alpha x\|.
\]
Moreover,
\[
    B^\dagger x=(I+A)^\alpha x,
    \quad x\in E_B.
\]

To check Assumption \ref{ass:fixed-time-integrability} we must show that
\[
    s\mapsto S(s)B
    =
    e^{-sA}(I+A)^{-\alpha}
\]
belongs to \(\gamma(L^2(0,t;E),E)\).
Since \(E\) is a Hilbert space, this is equivalent to the associated covariance
operator being trace class. This covariance operator is
\[
    Q_t
    =
    \int_0^t
    e^{-2sA}(I+A)^{-2\alpha}\dd s.
\]
By the spectral theorem,
\[
    Q_t=q_{\alpha,t}(A),
\]
where
\[
    q_{\alpha,t}(\lambda)
    =
    (1+\lambda)^{-2\alpha}
    \frac{1-e^{-2t\lambda}}{2\lambda},
    \quad \lambda>0,
\]
and \(q_{\alpha,t}(0):=t\).  Assumption \ref{ass:fixed-time-integrability}
takes the form
\[
    \Tr (q_{\alpha,t}(A))<\infty.
\]
If \(A\) has compact resolvent and eigenvalues
\((\lambda_n)_{n\ge1}\), counted with multiplicity, this condition becomes
\[
    \sum_{n\ge1}
    (1+\lambda_n)^{-2\alpha}
    \frac{1-e^{-2t\lambda_n}}{2\lambda_n}
    <\infty,
\]
with the usual interpretation of the summand at \(\lambda_n=0\).

We now compute the admissible directions. Let \(\mathsf E_{A}\) be
the spectral measure of \(A\). For \(\xi\in E\), define the finite
measure
\[
    \mu_\xi(\Delta)
    = \iprod{\mathsf E_{A}(\Delta)\xi}{\xi}_{E} =
    \|\mathsf E_{A}(\Delta)\xi\|^2,
    \qquad \Delta\in\mathscr B(\R_+).
\]
Then
\[
\begin{aligned}
    \int_0^t
    \|S(s)\xi\|_{E_B}^2\dd s
    & =
    \int_0^t
    \|(I+A)^\alpha e^{-sA}\xi\|^2
    \dd s
    =
    \int_{[0,\infty)}
    m_{\alpha,t}(\lambda)\dd\mu_\xi(\lambda),
\end{aligned}
\]
where
\[
    m_{\alpha,t}(\lambda)
    =
    (1+\lambda)^{2\alpha}
    \frac{1-e^{-2t\lambda}}{2\lambda},
    \quad \lambda>0,
\]
and \(m_{\alpha,t}(0):=t\). This gives the following description of
\(\mathscr D_t\).

If \(0\le\alpha\le1/2\), then \(m_{\alpha,t}\) is bounded on
\([0,\infty)\). Hence
\[
    \mathscr D_t=E,
\]
and there is a constant \(C_{\alpha,t}\) such that
\[
    \Bigl(
    \int_0^t
    \|S(s)\xi\|_{E_B}^2\dd s
    \Bigr)^{1/2}
    \le
    C_{\alpha,t}\|\xi\|,
    \quad \xi\in E.
\]
If \(\alpha>1/2\), then
\[
    m_{\alpha,t}(\lambda)
    \simeq_{\alpha,t}
    (1+\lambda)^{2\alpha-1},
    \quad \lambda\ge0.
\]
Consequently,
\[
    \mathscr D_t
    =
    \Dom((I+A)^{\alpha-1/2}),
\]
and the square function norm is equivalent to
\(\|(I+A)^{\alpha-1/2}\xi\|\).

The explicit Bismut--Elworthy--Li formula therefore gives the following concrete expression.
Let \(f:E\to\R\) be bounded and Borel measurable, let \(x\in E\), and let
\(\xi\in\mathscr D_t\). If \(a\in W^{1,\infty}(0,t)\) satisfies
\(a(0+)=0\) and \(a(t-)=1\), then
\[
    D_\xi P_tf(x)
    =
    \E\Bigl[
    f(U_t^x)
    \int_0^t
    a'(s)
    \iprod{(I+A)^\alpha e^{-sA}\xi}
       {\dd W_H(s)}_{H}
    \Bigr].
\]
Moreover, with the choice \(a(s) = s/t\),
\[
    |D_\xi P_tf(x)|
    \le
    \frac{\sqrt{2/\pi}}{t}\,\|f\|_\infty
    \Bigl(
    \int_0^t
    \|(I+A)^\alpha e^{-sA}\xi\|^2\dd s
    \Bigr)^{1/2}.
\]
Thus, for \(0\le\alpha\le1/2\), the semigroup \(P_t\) maps bounded Borel
functions into Lipschitz continuous functions on \(E\). For \(\alpha>1/2\),
the same formula gives directional smoothing in the fractional domain
\(\Dom((I+A)^{\alpha-1/2})\).

\subsection{The Dirichlet heat semigroup on \texorpdfstring{\(L^p(0,1)\)}{Lp(0,1)}}
\label{subsec:dirichlet-heat-lp-example}

We next consider a simple non-Hilbertian example. Let
\[
    E=L^p(0,1),
    \quad 1\le p\le2,
\]
and let \(S_p=(S_p(t))_{t\ge0}\) be the heat semigroup on \(E\)
generated by the Dirichlet Laplacian on \(L^p(0,1)\).
We take \(H=L^2(0,1)\) and let \(B:H\to E\)
be the natural inclusion mapping.

In this case
\[
    E_B=\Ran (B)=L^2(0,1),
\]
viewed as a Hilbertian subspace of \(L^p(0,1)\), and
\[
    \|y\|_{E_B}=\|y\|_{L^2(0,1)}.
\]
Moreover \(B^\dagger\) is just the identity map on \(L^2(0,1)\).

The Dirichlet heat semigroup is given by
\[
    S_p(t)f(r)
    =
    \int_0^1 k_t(r,u)f(u)\dd u,
    \quad f\in C_{\rm c}(0,1),\ r\in(0,1),
\]
where
\[
    k_t(r,u)
    =
    2\sum_{n=1}^\infty
    e^{-n^2\pi^2t}
    \sin(n\pi r)\sin(n\pi u),
    \quad r,u\in(0,1).
\]
Using the standard identification
\[
    \gamma\bigl(L^2(0,t;L^2(0,1)),L^p(0,1)\bigr)
    \simeq
    L^p\bigl(0,1;L^2((0,t)\times(0,1))\bigr),
\]
to check Assumption \ref{ass:fixed-time-integrability} it is enough to show
that
\begin{equation}\label{eq:sqfc-estimate}
    r\mapsto
    \Bigl(
    \int_0^t\int_0^1 |k_s(r,u)|^2\dd u\dd s
    \Bigr)^{1/2}
    \quad\hbox{belongs to } L^p(0,1).
\end{equation}
By symmetry and the semigroup identity,
\[
    \int_0^1 |k_s(r,u)|^2\dd u
    =
    k_{2s}(r,r).
\]
The eigenfunction expansion gives
\[
    k_{2s}(r,r)
    =
    2\sum_{n=1}^\infty
    e^{-2n^2\pi^2s}\sin^2(n\pi r)
    \le
    2\sum_{n=1}^\infty e^{-2n^2\pi^2s}
    \lesssim s^{-1/2},
    \quad s>0.
\]
Consequently,
\[
    \int_0^t\int_0^1 |k_s(r,u)|^2\dd u\dd s
    \lesssim
    \int_0^t s^{-1/2}\dd s
    <\infty,
\]
uniformly in \(r\in(0,1)\). This proves \eqref{eq:sqfc-estimate}.

The admissible directions are easy to determine.
Since \(E_B=L^2(0,1)\),
\[
    \mathscr D_t
    =
    \Bigl\{
    \xi\in L^p(0,1):
    S_p(\cdot)\xi\in L^2(0,t;L^2(0,1))
    \Bigr\}.
\]
For \(1\le p\le2\), the heat semigroup estimate
\[
    \|S_p(s)\xi\|_{L^2(0,1)}
    \le
    C_p\,s^{-\beta_p}\|\xi\|_{L^p(0,1)},
    \qquad
    \beta_p=\frac12\Bigl(\frac1p-\frac12\Bigr),
\]
gives
\[
    \int_0^t
    \|S_p(s)\xi\|_{L^2(0,1)}^2\dd s
    \le
    C_{p,t}^2\|\xi\|_{L^p(0,1)}^2,
\]
because \(2\beta_p<1\). Hence
\[
    \mathscr D_t=L^p(0,1),
    \quad 1\le p\le2.
\]
More explicitly, one checks that
\[
    C_{p,t}
    \lesssim_p
    t^{\,3/4-1/(2p)}.
\]

The Bismut--Elworthy--Li formula therefore holds in every direction
\(\xi\in L^p(0,1)\). For bounded Borel \(f:L^p(0,1)\to\R\), \(x,\xi\in L^p(0,1)\),
and \(a\in W^{1,\infty}(0,t)\) satisfying \(a(0+)=0\) and \(a(t-)=1\), one obtains
\[
    D_\xi P_tf(x)
    =
    \E\Bigl[
    f(U_t^x)
    \int_0^t
    a'(s)\iprod{S_p(s)\xi}{\dd W_{L^2(0,1)}(s)}_{L^2(0,1)}
    \Bigr].
\]
This implies the estimate
\[
    |D_\xi P_tf(x)|
    \le
    \sqrt{2/\pi}\,\|f\|_\infty
    \|a'\|_{L^\infty(0,t)}
    \Bigl(
    \int_0^t
    \|S_p(s)\xi\|_{L^2(0,1)}^2\dd s
    \Bigr)^{1/2}.
\]
With the choice \(a(s)=s/t\), this gives the explicit estimate
\[
    |D_\xi P_tf(x)|
    \le
    C_p\,t^{-1/4-1/(2p)}
    \|f\|_\infty\|\xi\|_{L^p(0,1)}.
\]
Thus, in this example, \(P_t\) maps bounded Borel functions on \(L^p(0,1)\) into
Lipschitz continuous functions for every \(t>0\) and every \(1\le p\le2\).

\section{Relation with the RKHS criterion}
\label{sec:relation-rkhs-criterion}

We briefly relate the preceding formula to the usual Cameron--Martin
description of linear Ornstein--Uhlenbeck semigroups. Let \(\mu_t\) be the law
of \(W_A(t)\), and let \(H_t\) be its reproducing kernel Hilbert space. In the
present notation,
\[
    H_t=\Ran (R_t),
    \qquad
    R_th=\int_0^t S(t-s)Bh(s)\dd s,
\]
with the quotient Hilbert norm inherited from \(L^2(0,t;H)\).

The classical Gaussian criterion says that differentiability of \(P_tf\) in the
direction \(\xi\) is governed by the inclusion
\[
    S(t)\xi\in H_t.
\]
Theorem \ref{thm:RKHS} is the corresponding derivative formula in the
present notation.

As we have seen in Proposition \ref{prop:linear-exact-control}, the square
function condition \(\xi \in \mathscr D_t\) used in this paper is a concrete
sufficient condition for the RKHS inclusion \(S(t)\xi\in H_t\).
In general, however, the sufficient condition does not characterise all directions for which
\(S(t)\xi\in H_t\), as is shown in the following example.

\begin{example}
Let \(E=\mathbb R^2\),
\(H=\mathbb R\), and let
\[
    Bh = h e_1, \quad h\in\mathbb R,
\]
where \(e_1=(1,0)\). Then \(E_B=\operatorname{span}\{e_1\}\). Let
\((S(s))_{s\in\mathbb R}\) be the rotation group,
\[
    S(s)=
    \begin{pmatrix}
    \cos s & -\sin s\\
    \sin s & \cos s
    \end{pmatrix}.
\]
For \(t>0\), the operator \(R_t:L^2(0,t)\to\mathbb R^2\) is given by
\[
    R_t u
    =
    \int_0^t S(t-s)e_1\,u(s)\,{\rm d}s .
\]
We claim that \(H_t=\Ran(R_t)=\mathbb R^2\). Indeed, if
\(y\in\mathbb R^2\) is orthogonal to \(\Ran(R_t)\), then
\(\iprod{y}{S(t-s)e_1}_{\R^2}=0\) for almost all \(s\in(0,t)\), and by continuity
it vanishes identically on \((0,t)\). Since the vectors \(S(t-s)e_1\), \(0<s<t\), span
\(\mathbb R^2\), it follows that \(y=0\). Thus
\(\Ran(R_t)\) is dense in \(\mathbb R^2\), and since it is a
finite-dimensional subspace, \(\Ran(R_t)=\mathbb R^2\).

Consequently \(S(t)\xi\in H_t\) for every \(\xi\in\mathbb R^2\). On the
other hand, the condition \(\xi\in\mathscr D_t\) requires
\[
    S(s)\xi\in E_B=\operatorname{span}\{e_1\}
    \ \ \hbox{for almost all }s\in(0,t).
\]
Writing \(\xi=(\xi_1,\xi_2)\), this means
\[
    \xi_1\sin s+\xi_2\cos s=0
    \ \ \hbox{for almost all }s\in(0,t).
\]
Again by continuity, this identity holds for all \(s\in(0,t)\), and therefore
\(\xi_1=\xi_2=0\). Hence
\[
    \mathscr D_t=\{0\},
    \qquad
    H_t=\mathbb R^2 .
\]
Thus the square function condition is only a sufficient condition for the
Cameron--Martin inclusion \(S(t)\xi\in H_t\), not a necessary one.
\end{example}

The following theorem gives a sufficient condition for equivalence.

\begin{theorem}[Equivalence with the RKHS criterion]
\label{thm:Dt-Ht-equivalence}
Fix \(t>0\). Suppose that the following two conditions hold.

\begin{enumerate}[leftmargin=*, label=(\roman*)]
\item[\rm(i)] The noise range is invariant under the semigroup up to time \(t\):
\[
    S(r)E_B\subseteq E_B,\quad 0\le r\le t,
\]
and the restricted operators \(S(r):E_B\to E_B\), \(0\le r\le t\), are
uniformly bounded.

\item[\rm(ii)] The semigroup has the following backward regularity property with
respect to \(E_B\): if \(\xi\in E\) and \(S(t)\xi\in E_B\), then
\(S(s)\xi\in E_B\) for almost all \(s\in(0,t)\)
and
\[
    \int_0^t \|S(s)\xi\|_{E_B}^2\,{\rm d}s<\infty .
\]
\end{enumerate}

Then for every \(\xi\in E\) we have
\[
    \xi\in\mathscr D_t
    \quad\Longleftrightarrow\quad
    S(t)\xi\in H_t .
\]
\end{theorem}

\begin{proof}
The implication
\(\Longrightarrow\) has already been proved above, and does not use the additional
assumptions. It remains to prove the converse.

Let \(R_t:L^2(0,t;H)\to E\) be the deterministic operator associated with
the stochastic convolution as before.
We first observe that assumption (i) implies
\[
    H_t=\Ran(R_t)\subseteq E_B.
\]
Indeed, if \(u\in L^2(0,t;H)\), then \(Bu(s)\in E_B\) for almost all \(s\in (0,t)\), and
\[
    S(t-s)Bu(s)\in E_B
    \ \ \hbox{for almost all }s\in(0,t).
\]
Moreover, by the uniform boundedness of the restrictions \(S(r):E_B\to E_B\),
\[
    \|S(t-s)Bu(s)\|_{E_B}
    \lesssim
    \|Bu(s)\|_{E_B}
    \le
    \|u(s)\|_{H}.
\]
Since \(u\in L^2(0,t;H)\subseteq L^1(0,t;H)\), the function
\(s\mapsto S(t-s)Bu(s)\) is Bochner integrable as an \(E_B\)-valued function.
Therefore
\[
    R_tu\in E_B.
\]
Thus \(\Ran(R_t)\subseteq E_B\), as claimed.

Now suppose that \(S(t)\xi\in H_t\). Since \(H_t\subseteq E_B\), we have
\(S(t)\xi\in E_B\). Assumption (ii) then gives
\(S(s)\xi\in E_B\) for almost all \(s\in(0,t)\)
and
\[
    \int_0^t \|S(s)\xi\|_{E_B}^2\,{\rm d}s<\infty .
\]
Thus \(S(t)\xi\in H_t\) implies \(\xi\in\mathscr D_t\), and the proof is complete.
\end{proof}

Let us indicate how this criterion relates to the examples of
Section \ref{sec:examples-first-consequences}. In the self-adjoint Hilbert
space example, with \(B=(I+A)^{-\alpha}\), the space
\[
    E_B=\Dom((I+A)^\alpha)
\]
is invariant under \(S(r)=e^{-rA}\), and the restrictions of \(S(r)\) to
\(E_B\) are contractions. Thus condition {\rm (i)} of
Theorem \ref{thm:Dt-Ht-equivalence} holds. Condition {\rm (ii)}, however,
is equivalent in this example to the assertion that
\[
    S(\cdot)\xi\in L^2(0,t;E_B)
\]
whenever \(S(t)\xi\in E_B\). By the computation in
Subsection \ref{subsec:self-adjoint-hilbert-example},
in the typical unbounded case this holds for all \(\xi\in E\) precisely in
the range \(0\le\alpha\le1/2\). In that range we have \(\mathscr D_t=E\), and the theorem gives
\(S(t)\xi\in H_t\) for all \(\xi\in E\).
For \(\alpha>1/2\) the theorem does not apply, and the converse implication
fails in general: the analytic smoothing of \(S(t)\) may map \(S(t)\xi\) into
\(H_t\) even though the square function
\[
    \int_0^t \|S(s)\xi\|_{E_B}^2\,{\rm d}s
\]
is infinite.

In the Dirichlet heat example on \(L^p(0,1)\), \(1\le p\le2\), one has
\(E_B=L^2(0,1)\), and the heat semigroup leaves \(L^2(0,1)\) invariant
contractively. Hence condition {\rm (i)} holds. Moreover, the heat kernel
estimate used in Subsection \ref{subsec:dirichlet-heat-lp-example} gives
\[
    S_p(\cdot)\xi\in L^2(0,t;L^2(0,1)),
    \quad \xi\in L^p(0,1).
\]
Since \(S_p(t)\xi\in L^2(0,1)\) for every \(t>0\) and every
\(\xi\in L^p(0,1)\), condition {\rm (ii)} holds as well. Thus the theorem
applies and recovers our findings that \(\mathscr D_t=L^p(0,1)\) and \(S_p(t)\xi\in H_t\) for all \(\xi\in L^p(0,1)\).

\bigskip
\noindent
{\em Acknowledgement and AI disclosure} --
A first sketch for this paper was written approximately 15 years ago during a
visit of the author to Bohdan Maslowski, while working on our joint paper
\cite{MaslowskiVanNeerven2012}. The invitation to contribute to Bohdan's 70th
birthday volume provided an excellent opportunity to return to the project.

The author used GPT 5.5 as a tool for mathematical brainstorming, restructuring,
and editorial assistance while preparing the manuscript. In particular the
elegant construction involving the space \(E_B\) was suggested by GPT.
All arguments and the final text were checked and revised by the author.

\bibliographystyle{plainnat}
\bibliography{BEL}

\begin{thebibliography}{14}
\providecommand{\natexlab}[1]{#1}
\providecommand{\url}[1]{\texttt{#1}}
\expandafter\ifx\csname urlstyle\endcsname\relax
  \providecommand{\doi}[1]{doi: #1}\else
  \providecommand{\doi}{doi: \begingroup \urlstyle{rm}\Url}\fi

\bibitem[Bismut(1984)]{Bismut1984}
J.-M. Bismut.
\newblock \emph{Large Deviations and the {M}alliavin Calculus}, volume~45 of
  \emph{Progress in Mathematics}.
\newblock Birkh{\"a}user Boston, Boston, MA, 1984.

\bibitem[Bogachev(1998)]{Bogachev1998}
V.I. Bogachev.
\newblock \emph{Gaussian Measures}, volume~62 of \emph{Mathematical Surveys and
  Monographs}.
\newblock American Mathematical Society, Providence, RI, 1998.

\bibitem[Chojnowska-Michalik and Goldys(2002)]{ChojnowskaMichalikGoldys2002}
A.~Chojnowska-Michalik and B.~Goldys.
\newblock Symmetric {O}rnstein--{U}hlenbeck semigroups and their generators.
\newblock \emph{Probab. Theory Relat. Fields.}, 124:\penalty0 459--486, 2002.

\bibitem[Da~Prato and Zabczyk(1992)]{DaPratoZabczyk1992}
G.~Da~Prato and J.~Zabczyk.
\newblock \emph{Stochastic Equations in Infinite Dimensions}, volume~44 of
  \emph{Encyclopedia of Mathematics and its Applications}.
\newblock Cambridge University Press, Cambridge, 1992.

\bibitem[Da~Prato and Zabczyk(1996)]{DaPratoZabczyk1996}
G.~Da~Prato and J.~Zabczyk.
\newblock \emph{Ergodicity for Infinite Dimensional Systems}, volume 229 of
  \emph{London Mathematical Society Lecture Note Series}.
\newblock Cambridge University Press, Cambridge, 1996.

\bibitem[Elworthy and Li(1994)]{ElworthyLi1994}
K.D. Elworthy and X.-M. Li.
\newblock Formulae for the derivatives of heat semigroups.
\newblock \emph{J. Funct. Anal.}, 125\penalty0 (1):\penalty0 252--286, 1994.

\bibitem[Goldys and Peszat(2024)]{GoldysPeszat2024}
B.~Goldys and Sz. Peszat.
\newblock Differentiability of transition semigroup of generalized
  {O}rnstein--{U}hlenbeck process: a probabilistic approach.
\newblock \emph{arXiv2410.20074}, 2024.

\bibitem[Goldys and van Neerven(2003)]{GoldysVanNeerven2003}
B.~Goldys and J.M.A.M. van Neerven.
\newblock Transition semigroups of {B}anach space valued
  {O}rnstein--{U}hlenbeck processes.
\newblock \emph{Acta App. Math.}, 76:\penalty0 283--330, 2003.

\bibitem[Hyt{\"o}nen et~al.(2017)Hyt{\"o}nen, van Neerven, Veraar, and
  Weis]{HNVW-volume2}
T.P. Hyt{\"o}nen, J.M.A.M. van Neerven, M.C. Veraar, and L.W. Weis.
\newblock \emph{Analysis in {B}anach Spaces. {V}olume II}, volume~67 of
  \emph{Ergebnisse der Mathematik und ihrer Grenzgebiete. 3. Folge. A Series of
  Modern Surveys in Mathematics}.
\newblock Springer, Cham, 2017.

\bibitem[Maslowski and van Neerven(2013)]{MaslowskiVanNeerven2012}
B.~Maslowski and J.M.A.M. van Neerven.
\newblock Equivalence of laws and null controllability for {SPDE}s driven by a
  fractional {B}rownian motion.
\newblock \emph{NoDEA}, 20\penalty0 (4):\penalty0 1473--1498, 2013.

\bibitem[{\VAN{Neerven}{van}{van}}~Neerven(1998)]{vanNeerven1998}
J.M.A.M. {\VAN{Neerven}{van}{van}}~Neerven.
\newblock Non-symmetric {O}rnstein--{U}hlenbeck semigroups in {B}anach spaces.
\newblock \emph{J. Funct. Anal.}, 1998.
\newblock Final version dated January 26, 1998.

\bibitem[{\VAN{Neerven}{van}{van}}~Neerven(2010)]{vanNeerven2010}
J.M.A.M. {\VAN{Neerven}{van}{van}}~Neerven.
\newblock \(\gamma\)-{R}adonifying operators---a survey.
\newblock \emph{Proc. CMA}, 44:\penalty0 1--62, 2010.

\bibitem[Nualart(2006)]{Nualart2006}
D.~Nualart.
\newblock \emph{The {M}alliavin Calculus and Related Topics}.
\newblock Springer, Berlin, second edition, 2006.

\bibitem[Peszat and Zabczyk(1995)]{PeszatZabczyk1995}
Sz. Peszat and J.~Zabczyk.
\newblock Strong {F}eller property and irreducibility for diffusions on
  {H}ilbert spaces.
\newblock \emph{The Annals of Probability}, 23\penalty0 (1):\penalty0 157--172,
  1995.

\end{thebibliography}

\end{document}